\documentclass[12pt]{article}

\usepackage{amsmath, amsthm, amssymb, enumerate}
\usepackage{color}
\usepackage{datetime}
\usepackage{fancyhdr}
\chardef\coloryes=1 %%%out
\chardef\isitdraft=1 %%%out
\ifnum\isitdraft=1 %%%out
   \def\version{9} %%%out
   \def\eqref#1{({\ref{#1}})}                %saves writing paranthesis%%%out

\definecolor{labelkey}{gray}{.3}%%%out

\definecolor{refkey}{rgb}{.3,0.3,0.3}%%%out

\else%%%out

  \def\startnewsection#1#2{\section{#1}\label{#2}\setcounter{equation}{0}}   %\starts a new section
  \def\nnewpage{} %\nnewpage does nothing
\fi%%%out

\begin{document}
%\newcommand{\llabel}{\label}             %\llabel is just a synonim for \label
%\newcommand{\rref}{\ref}                 %\rref is  just a synonim for \ref
%\newcommand{\ccite}{\cite}               %\ccite is just a synonim for
                                         %when ref. to equations
\def\ques{{\cor \underline{??????}\cob}}
\def\nto#1{{\coC \footnote{\em \coC #1}}}
\def\fractext#1#2{{#1}/{#2}}
\def\fracsm#1#2{{\textstyle{\frac{#1}{#2}}}}   %smaller version of frac
\def\nnonumber{}

%\newcommand{\bv}{{u}}

%if \isitdraft=0 or \coloryes=0%%%out
\def\cor{{}}%%%out
\def\cog{{}}%%%out
\def\cob{{}}%%%out
\def\coe{{}}%%%out
\def\coA{{}}%%%out
\def\coB{{}}%%%out
\def\coC{{}}%%%out
\def\coD{{}}%%%out
\def\coE{{}}%%%out
\def\coF{{}}%%%out
%\def\MR#1{}%%%out

%%3
%\ifnum\isitdraft=1%%%out
\ifnum\coloryes=1%%%out

  \definecolor{coloraaaa}{rgb}{0.1,0.2,0.8}%%%out
  \definecolor{colorbbbb}{rgb}{0.1,0.7,0.1}%%%out
  \definecolor{colorcccc}{rgb}{0.8,0.3,0.9}%%%out
  \definecolor{colordddd}{rgb}{0.0,.5,0.0}%%%out
  \definecolor{coloreeee}{rgb}{0.8,0.3,0.9}%%%out
  \definecolor{colorffff}{rgb}{0.8,0.3,0.9}%%%out
  \definecolor{colorgggg}{rgb}{0.5,0.0,0.4}%%%out

 \def\cog{\color{colordddd}}%%%out
 \def\cob{\color{black}}%%%out
 \def\cor{\color{red}}%%%out
 \def\coe{\color{colorgggg}}%%%out

 \def\coA{\color{coloraaaa}}%%%out
 \def\coB{\color{colorbbbb}}%%%out
 \def\coC{\color{colorcccc}}%%%out
 \def\coD{\color{colordddd}}%%%out
 \def\coE{\color{coloreeee}}%%%out
 \def\coF{\color{colorffff}}%%%out
 \def\coG{\color{colorgggg}}%%%out

%%4
\fi%%%out
%\fi%%%out
\ifnum\isitdraft=1%%%out
   \chardef\coloryes=1 %%%out
   \baselineskip=17pt%%%out
\pagestyle{myheadings}
\reversemarginpar

\def\const{\mathop{\rm const}\nolimits}  %const (for a constant)
\def\diam{\mathop{\rm diam}\nolimits}    %diameter

 \def\llabel#1{\label{#1}{\ \mbox{\rm\color{red} {#1}\color{black}}}}

\def\rref#1{{\ref{#1}{\rm \tiny \fbox{\tiny #1}}}}
\def\theequation{\fbox{\bf \thesection.\arabic{equation}}}
\def\ccite#1{{\cite{#1}{\rm \tiny ({#1})}}}

\def\startnewsection#1#2{\newpage\cog \section{#1}\cob\label{#2}

\setcounter{equation}{0}
%\rfoot{\thepage}
\pagestyle{fancy}

\lhead{\cob Section~\ref{#2}, #1 }
\cfoot{}
\rfoot{\thepage}
\lfoot{\cob{\today,~\currenttime}~(c75-iklt2, Version~\fbox{\version})}}
%\lhead{\cob Section~\rref{#2}, #1 \cob{\coe\today,~\currenttime}~(c60-sd, Version~\version)}}
\chead{}
\rhead{\thepage}
\def\nnewpage{\newpage}

%change rref to ref
%change llabel to label
%change draft to final
%comment out macros
%check periods
%check undefined references
%check uncalled formulas
%check date
%change startnewsection to section
%counters:
\newcounter{startcurrpage}
\newcounter{currpage}

\def\llll#1{{\rm\tiny\fbox{#1}}}
%%%out
   \def\blackdot{{\color{red}{\hskip-.0truecm\rule[-1mm]{4mm}{4mm}\hskip.2truecm}}\hskip-.3truecm}%%%out
   \def\bdot{{\coC {\hskip-.0truecm\rule[-1mm]{4mm}{4mm}\hskip.2truecm}}\hskip-.3truecm}%%%out
   \def\purpledot{{\coA{\rule[0mm]{4mm}{4mm}}\cob}}%%%out
   \def\pdot{\purpledot}%%%out
\else%%%out  
   \baselineskip=15pt
   \def\blackdot{{\rule[-3mm]{8mm}{8mm}}}%%%out
   \def\purpledot{{\rule[-3mm]{8mm}{8mm}}}%%%out
   \def\pdot{}
%%5
\fi%%%out

\def\tdot{\fbox{\fbox{\bf\tiny I'm here; \today \ \currenttime}}}
\def\nts#1{{\hbox{\bf ~#1~}}} %nts=note to self
\def\nts#1{{\cor\hbox{\bf ~#1~}}} %nts=note to self%%%out
\def\ntsf#1{\footnote{\hbox{\bf ~#1~}}} %nts=note to self
\def\ntsf#1{\footnote{\cor\hbox{\bf ~#1~}}} %nts=note to self%%%out
\def\bigline#1{~\\\hskip2truecm~~~~{#1}{#1}{#1}{#1}{#1}{#1}{#1}{#1}{#1}{#1}{#1}{#1}{#1}{#1}{#1}{#1}{#1}{#1}{#1}{#1}{#1}\\}%%%out
\def\biglineb{\bigline{$\downarrow\,$ $\downarrow\,$}}%%%out
\def\biglinem{\bigline{---}}%%%out
\def\biglinee{\bigline{$\uparrow\,$ $\uparrow\,$}}%%%out

\newtheorem{Theorem}{Theorem}[section]
\newtheorem{Corollary}[Theorem]{Corollary}
\newtheorem{Proposition}[Theorem]{Proposition}
\newtheorem{Lemma}[Theorem]{Lemma}
\newtheorem{Remark}[Theorem]{Remark}
\newtheorem{Example}[Theorem]{Example}
\newtheorem{Claim}[Theorem]{Claim}
\newtheorem{Assumption}[Theorem]{Assumption}
\newtheorem{definition}{Definition}[section]
\def\theequation{\thesection.\arabic{equation}}
\def\endproof{\hfill$\Box$\\}
\def\square{\hfill$\Box$\\}
\def\comma{ {\rm ,\qquad{}} }            %comma in a formula
\def\commaone{ {\rm ,\qquad{}} }         %second comma in a formula
\def\dist{\mathop{\rm dist}\nolimits}    %distance
\def\sgn{\mathop{\rm sgn\,}\nolimits}    %sgn
\def\Tr{\mathop{\rm Tr}\nolimits}    %trace
\def\div{\mathop{\rm div}\nolimits}    %divergence
\def\supp{\mathop{\rm supp}\nolimits}    %divergence
\def\divtwo{\mathop{{\rm div}_2\,}\nolimits}    %two dimensional divergence
\def\re{\mathop{\rm {\mathbb R}e}\nolimits}    %distance
\def\div{\mathop{\rm{Lip}}\nolimits}   %Lip
\def\indeq{\qquad{}}                     %indentation in formulas
\def\period{.}                           %period in a formula
\def\semicolon{\,;}                      %semicolon in a formula
%**end of header

%A Galerkin approximation for the stochastic Navier-Stokes equations
\title{Dynamic optimal contract under parameter uncertainty with risk averse agent and principal}
\author{Kerem U\u{G}URLU}
\maketitle

\date{}

\begin{center}
\end{center}

\medskip

\indent Department of Applied Mathematics, University of Washington, Seattle, WA 98195\\
\indent e-mails:keremu@uw.edu

\begin{abstract}
We consider a continuous time Principal-Agent model on a finite time horizon, where we look for the existence of an optimal contract both parties agreed on. Contrary to the main stream, where the principal is modelled as risk-neutral, we assume that both the principal and the agent have exponential utility, and are risk averse with same risk awareness level. Moreover, the agent's quality is unknown and modeled as a filtering term in the problem, which is revealed as time passes by. The principal can not observe the agent's real action, but can only recommend action levels to the agent. Hence, we have a \textit{moral hazard} problem. In this setting, we give an explicit solution to the optimal contract problem.     
\end{abstract}

%\noindent\thanks{\em Mathematics Subject Classification\/}:
%35R35, %Free boundary problems
%35Q30, %Stokes and Navier-Stokes equations
%76D05  %Navier-Stokes equations 

\noindent\thanks{\em Keywords: \/}dynamic  principal agent problem, moral hazard, optimal control.

\section{Introduction}

We consider optimal contracting between two parties, principal (``she") and agent (``he") in continuous time, when agent's actual effort can not be observed by the principal. In economics, this type of problems is called ``hidden action" or ``moral hazard" problem, where the agent's control of the drift of the output process can not be contracted upon. To give an example for the moral hazard problem, we can consider a scenario, where the investor (``principal) hires a portfolio manager (``agent") to manage her savings. The investor can not observe the actual effort (or action) of the portfolio manager but only the current wealth of the portfolio. Hence, in case the investor is not satisfied about the performance of the portfolio manager, the manager could blame the market and argue he gave the best performance for her savings, since the investor can not observe the actual efforts of her portfolio manager, anyway.

The seminal paper on the continuous time principal-agent problem is \cite{HM87}, where both parties have exponential utilities and agree on a linear optimal contract. Their results are generalized and extended by several authors, ( see e.g. \cite{SS93, SS97, HS02, CWZ09, CZ07, M98, M00} among others). A nice survey of the literature is provided by Sung in \cite{S01}. Recently, \cite{CDT16} has considered a general formulation of the principal-agent problem with a lump-sum payment on a finite horizon, where the agent influences both the agent and the volatility of the output, where the proofs use techniques based on Backward Stochastic Differential Equations approach to non-Markovian stochastic control.  

In another seminal paper, \cite{S08} works in continuous time moral hazard model with infinite horizon and the payments are paid continuously, rather than as a lump sum payment at the terminal time. In \cite{S08}, the principal is risk neutral and the agent is risk averse and the agent only controls the drift of the output. \cite{PJ15} extends \cite{S08} to the case of unobserved drift and makes use of the Stochastic Maximum Principle. \cite{N14} solves a principal-agent moral hazard problem, where both the principal and the agent are risk averse, the payments are continuous, and the agent controls the unobserved drift term of the output. The main difference in the contract between a model with a risk averse and a risk neutral principal is whether there is an ongoing consumption and dividends, as would be the case with a risk aversion, or only ``lumpy" consumption and dividends which would be the case with risk-neutrality. While most of the literature focuses on a risk neutral principal, there are exceptions like the seminal work \cite{HM87} and the recent work \cite{N14}. The current manuscript is another work in this direction with a risk averse principal.

In this paper, we solve a moral hazard problem in continuous time Brownian model, where there is an additional endogenous learning term representing agent's unknown quality added to the model. To represent unknown quality of the agent, we follow the framework introduced in \cite{PJ15}. However, we do not assume that the principal is risk-neutral as in \cite{PJ15}, but instead both the principal and the agent have exponential utility with same risk-awareness level as in \cite{N14}. Our model is both quantitatively and qualitatively different from the one in \cite{PJ15} and from \cite{N14}. In the risk neutral case as in \cite{PJ15}, contracting is profitable, because the principal can extract profits by providing insurance to the agent. In our case, combined with the unknown quality of the agent, the interaction is more complicated. We can not conclude directly that the economic benefits will decrease as the difference in risk aversion between the principal and the agent shrinks, since we don't know the quality of the agent a priori, but his quality is revealed with time. The same obstacle has been observed in \cite{PJ15}, where their results hold conditioned enough time has passed to conclude qualitative results about the model. In our model, we see that as the unknown quality of the agent is revealed with time such that it does not affect the model significantly, our model converges to the analogous findings without the parameter uncertainty. On the other hand, even though the risk awareness if the principal is taken into consideration in \cite{N14}, there is no endogenous learning term, hence there are no aforementioned complications related to it. Furthermore, we also show that both parties agree on a contract, where the agent gives full effort from the beginning of the contract until the horizon $T$. Hence, neither the risk awareness of the principal nor the unknown quality of the agent do not affect the agent's actual effort level given throughout the model, even though the principal can not observe it, which was to be observed in \cite{PJ15}, whereas \cite{N14} focused on interior optimal efforts. 

The rest of the paper is as follows. In Section 2, we outline the general model of the problem. In Section 3, we describe the agent's problem and find the dynamics of the continuation value function of the agent. In Section 4, we solve the principal's optimal control problem, describe the optimal contract, and in Section 5 we further discuss and elaborate our main results and conclude the paper.  
\section{The Model}
In this section, we give the framework and dynamics of the model. Let $\{W_t\}_{t \geq 0}$ be a standard Brownian motion on a probability space $(\Omega,\mathcal{F},\mathbb{P})$, where $\mathcal{F}_t$ is generated by the Brownian motion $W_t$. As in \cite{PJ15}, we assume that the cumulative output $y_t$ up to time $T$ satisfies the stochastic integral equation
\begin{equation}
\label{main_eqn}
y_t = \int_0^t \big( \eta + a_s \big) ds + \int_0^t \sigma dW_s,
\end{equation}
for $0 \leq t \leq T$, where $\eta$ stands for the quality of the agent and is denoted by $\eta$, and $a_t \in [0,M]$ for $0 \leq t \leq T$ is the effort provided by the agent. $\eta$ is unknown and we model it as in \cite{PJ15} with the common prior being normal with mean $m_0$ and precision $h_0$. Posterior over $\eta$, denoted by $\hat{\eta}$, depends on $y_t$ and on cumulative effort $A_t \triangleq \int_0^t a_s ds$. Conditional on $(y_t,A_t,t)$, posterior belief about $\eta$ is also normal with
\begin{align}
\label{eqn22}
\hat{\eta}(y_t - A_t,t) &\triangleq \mathbb{E}_t[\eta| y_t,A_t] \\
\hat{\eta}(y_t - A_t,t) &\triangleq \frac{h_0m_0 + \sigma^{-2}(y_t -A_t)}{h_t}
\nonumber \\
h_t &\triangleq h_0 + \sigma^{-2}t
\nonumber \\
\hat{\eta}(0,0) &= m_0 \nonumber.
\end{align}
The principal does not observe the agent's effort, but can only recommend actions $\hat{\bar{a}}$. We denote the filtration generated by output and recommended actions $(\bar{y},\hat{\bar{a}})$ as 
\begin{equation}
\mathcal{F}_t^y \triangleq \sigma(y_s, \hat{\bar{a}}_s; 0\leq s\leq t) 
\end{equation}
and $\mathbb{F}^y \triangleq \{ \mathcal{F}_t^y \}_{t \geq 0}$, the $\mathbb{P}$-augmentation of this filtration. As in \cite{PJ15}, we take the utility function of the agent with $\lambda \in (0,1), \theta > 0$ as
\begin{equation}
\label{eqn25}
u(w,a) = -e^{-\theta w + \theta \lambda a}
\end{equation}
and actions of the agent are limited in a compact set $a_t \in [0,M]$ for $0 \leq t \leq T$. 
On the other hand, the agent knows the actual level of effort $\bar{a}$, which only he knows. Hence, the agent's information is more than the principal. We denote the filtration generated by output, recommended  actions and actual actions up to time $t$ as $(\bar{y},\bar{a}^*)$ as 
\begin{equation}
\mathcal{F}_t^a \triangleq \sigma(y_s, \bar{a}_s, \bar{a}^*; 0\leq s\leq t)
\end{equation}
and $\mathbb{F}^a \triangleq \{ \mathcal{F}_t^a \}_{0 \leq t \leq T}$, the $\mathbb{P}$-augmentation of this filtration. The agent is restricted to the class of control processes $\mathcal{A} \triangleq \{ a_t: [0,T] \times \Omega \rightarrow [0,M] \}$ that are $\mathbb{F}^a$-predictable. 
We work with the induced distributions on the space of continuous functions. We take the sample space $\Omega$ to be the space of all continuous paths $C[0,T]$ equipped with the supremum norm $\lVert \cdot \rVert_{\infty}$. On $C[0,T]$, we let $W_t^0 = \omega(t)$ be the family of coordinate functions and $\mathcal{F}_t^0 = \sigma\{ W_s^0, s \leq t \}$ the filtration generated by $W_t^0$. We denote by $P^0$ the corresponding Wiener measure on $(\Omega, \mathcal{F}_t^0)$ and let $\mathcal{F}_t$ be the completion with the null sets of $\mathcal{F}_T^0$. On this space, we define the corresponding Brownian motion $W_t^0$ as in Equation \eqref{main_eqn}. The set of admissible contracts $\mathcal{C}$ is the set of $\mathbb{F}^Y \triangleq \{ \mathcal{F}_t^Y\}_{0 \leq t \leq T}$ predictable functions $(a_t,w_t): [0,T] \times Y \rightarrow A \times W$. Hence, the contract specifies a wage $\bar{w}_t$ and a recommended action $\bar{a}_t$ at date $t$ that depend on the whole past history of the output $\bar{y}_t$. Then, given a contract, the agent takes his own choice of action $a_t$ at each time $t$. Thus, the set of admissible actions $\mathcal{A}$ for the agent are those $\mathcal{F}_t^a$-predictable functions $(\bar{a},\bar{w}): [0,T] \times Y \rightarrow A \times W$. 
\begin{definition}\label{def21} A contract is called implementable if the agent agrees to the contract at time zero and chooses the recommended actions: $(a^*,w^*) = (\bar{a},\bar{w})$.
\end{definition}
The dependence on the whole past implies that we can not use a direct approach to the agent's problem, since the entire past history $\bar{y}$ would be a state variable. To overcome this difficulty, we make the problem tractable as in \cite{B73,B78},\cite{CWZ09}, \cite{N14} by taking the key state variable to be the density of the output process rather than the output process itself. By considering different action choices corresponding to different output processes, we take the relative density process $\Gamma_t$, as defined in Equation \eqref{eqn211}. For $\sigma > 0$, we define
\begin{equation}
\label{eqn17}
dy_t = \sigma dW_t^0,
\end{equation}
where $y_0$ is given. This is the evolution of output under an effort policy $\bar{a}^0$, which makes the drift of output equal to zero at each time $t \in [0,T]$. Hence, different effort choices alter the evolution of output by changing the distribution over outcomes in $y$ with their corresponding $\Gamma_t$. 
\section{The Agent's Problem} We impose a terminal date $T$ on the contracting horizon. Until time $T$, both the principal and the agent are committed to the contract. To have an incentive compatible contract, we need to specify what action the agent would choose when facing a given contract. First, we assume the following assumption.
The agent's continuation value is 
\begin{equation}
\label{eqn38}
v(a,t) \triangleq \mathbb{E}\bigg[ \int_t^T e^{-\rho(s-t)}u(w_s,a_s)ds + e^{-\rho(T-t)}g(w_T) | \mathcal{F}_t^a \bigg],
\end{equation} where $\bar{y}_t \triangleq \{ y_s; 0 \leq s \leq t \}$ is the output history, $\rho \in (0,1)$ is the constant discount rate and the functions $u$ and $g$ are defined to be below. The history dependence on the past makes it necessary to change the relevant state variable.
We denote 
\begin{equation}
f(t,\bar{y},a_t) = \hat{\eta} (y_t - A_t,t) + a_t, 
\end{equation}
where 
\begin{align} 
A_t &= \int_0^t a_s ds,
\end{align}
here we recall that $\bar{y}_t$ means $y$ depends on the whole path. We denote the density depending on action $\bar{a}$ of $\mathcal{F}_t$-predictable processes:
\begin{equation}
\label{eqn211}
\Gamma_t(\bar{a}) = \exp \big( \int_0^t \sigma^{-1}f(s,\bar{y},a_s)dW_s^0 - \frac{1}{2}\int_0^t \lvert \sigma^{-1}f(s,\bar{y},a_s) \rvert^{2} ds \big),
\end{equation}
where $W_t^0$ is as defined in Equation \eqref{eqn17}. $\Gamma_t$ is an $\mathcal{F}_t$-martingale (as the assumptions on $f$ ensures that Novikov's condition is satisfied) with $\mathbb{E}[\Gamma_T(\bar{a}) ] = 1$ for all $\bar{a} \in \mathcal{A}$ where $\mathcal{A}$ stands for the set of admissible actions. Thus, by Girsanov theorem, we define a new measure $P^{\bar{a}}$ via:
\begin{equation}
\frac{dP^{\bar{a}}}{dP^0} = \Gamma_T(\bar{a}),
\end{equation}
and  by the filtering theorem of Fujisaki \cite{F72}, the process $W_t^{\bar{a}}$ is defined by 
\begin{equation}
W_t^{\bar{a}} = W_t^0 - \int_0^t \bar{\sigma}^{-1}f(s,\bar{y},a_s)ds
\end{equation}
is a Brownian motion under $P_{\bar{a}}$. Thus we have 
\begin{align}
dy_t &= \sigma dW_t^0 \\
&= \sigma [ dW_t^{\bar{a}} + \sigma^{-1}f(t,\bar{y},a_t)dt ] \\
&= f(t,\bar{y},a_t)dt + \sigma dW_t^{\bar{a}}
\end{align}
Hence each effort choice $\bar{a}$ results in a different Brownian motion $W_t^a$. $\Gamma_t$ defined above satisfies $\Gamma_t = \mathbb{E}[ \Gamma_T |\mathcal{F}_t^{\bar{a}}]$. Moreover, via derivation as in \cite{PJ15} by Ito lemma we have $\hat{\eta}$ is a $P^{\bar{a}}$-martingale with decreasing variance 
\begin{equation}\label{filt_mart}
d\hat{\eta}(y_t -A_t,t) = \frac{\sigma^{-1}}{h_t}dW_t^{\bar{a}}
\end{equation}
Using the state variable as the density process $\Gamma_t$, we rewrite the optimization problem as 
\begin{equation}  
v(a,t) = \mathbb{E}_t^0\bigg[ \int_t^T \Gamma_{s,T}^a e^{-\rho(s-t)}u(w_s),a_s)ds + e^{-\rho(T-t)}g(w_T) | \mathcal{F}_t^0 \bigg],
\end{equation}
where the terminal value function being $g(w_T) = -e^{\frac{1-\rho}{r}-\lambda \theta r w_T}$ of the agent is to be derived below.
This approach makes our optimization problem as tractable with optimal control techniques. The agent's problem then is to solve 
\begin{equation}
v^*(t) = \sup_{\bar{a} \in \mathcal{A}}v(\bar{a},\bar{w}).
\end{equation}
\begin{Theorem} \label{hamilton} For each fixed action process $a(\cdot)$, there exists a unique decomposition for the agent's continuation value Equation \eqref{eqn38} that satisfies
\begin{align}
\label{decomposition}
dv_t &= [\rho v_t - u(w_t,a_t)]dt + \sigma \gamma_t^a dW_t^{\bar{a}} \\
v_T &= g(w_T),
\end{align}
for some square integrable process $\gamma_t^a$, namely $\mathbb{E}^a[\int_0^T (\gamma_t^a)^2 dt] < \infty$. The process $\gamma_t^a$ is denoted as ``incentive compatibility parameter" in moral hazard literature (see e.g. \cite{CZ12}).  
\end{Theorem}
\begin{proof} Recall for each action $a$ we have for that action $a$
\begin{align} 
v(a,t) &= \mathbb{E}_t^a\bigg[ \int_t^T e^{-\rho(s-t)}u(w_s,a_s)ds + e^{-\rho(T-t)}g(w_T) | \mathcal{F}_t^a \bigg]
\\ \nonumber
&= e^{\rho t}\mathbb{E}_t^a\bigg[ \int_0^T e^{-\rho s}u(w_s,a_s)ds + e^{-\rho(T-t)}g(w_T)) | \mathcal{F}_t^a \bigg] 
\\ \nonumber & \indeq
- e^{\rho t}\int_0^t e^{-\rho s} u(w_s,a_s)ds
\\ \nonumber
dv(a,t) &= \rho v_t dt + \gamma_t^a \sigma dW_t^a - u(w_t,a_t)dt,
\end{align}
where we appeal to the Martingale representation theorem by Fujisaki \cite{F72} for square integrable martingales. We note here that 
\begin{equation}
\mathbb{E}_t^a\bigg[ \int_0^T e^{-\rho(s-t)}u(w_s,a_s)ds + e^{-\rho(T-t)}g(w_T) | \mathcal{F}_t^a \bigg],
\end{equation}
is a square integrable martingale since the functions $g(\cdot), u(\cdot,\cdot)$ are bounded for $a\in [0,M]$ and $t \in [0,T]$.
\end{proof}
Next, we characterize the necessary and sufficient conditions to maximize the value function of the agent in Equation \eqref{eqn38}. Our result is analogous to Proposition 4 in \cite{PJ15}.  
\begin{Lemma} Maximizing the Hamiltonian defined as
\begin{equation}
H(t,y,a,A,\gamma) \triangleq u(w_t,a_t) + \gamma( \hat{\eta}(A_t,y) + a_t ) 
\end{equation}
is sufficient for the agent to maximize his value function as in Equation \eqref{eqn38}. Furthermore, it is necessary for the incentive compatibility parameter $\gamma_t^a$ to satisfy 
\begin{equation}
\gamma_t^a(a,w) = -u_a(w_t,a_t^*) + \frac{\sigma^{-2}}{h_t}p_t,
\end{equation}
where the term $p_t$ is as defined in Equation \eqref{eqn332}.

\end{Lemma}
\begin{proof} By integrating Equation \eqref{decomposition} for the optimal action $\hat{a}$ and for any other action $\bar{a}$, we have 
\begin{align} 
e^{-\rho T}v(T) &= e^{-\rho T}g(w_T) = e^{-\rho t}v(t,\hat{a}) - \int_t^T e^{-\rho s}u(w_s,\hat{a}_s)ds + \int_t^T \hat{\zeta}_s\sigma dW_s^{\hat{a}}
\nonumber \\
e^{-\rho T}v(T) &= e^{-\rho T}g(w_T) = e^{-\rho t}v(t,\bar{a}) - \int_t^T e^{-\rho s}u(w_s,\bar{a}_s)ds + \int_t^T \bar{\zeta}_s\sigma dW_s^{\bar{a}}, 
\end{align}
where $\zeta_t \triangleq e^{-\rho t}\gamma_t$. Moreover, we have
\begin{align}
dy_t &= \sigma dW_t^0
\nonumber \\
dW_t^{\hat{a}} &= dW_t^0 - \frac{1}{\sigma}(\hat{\eta}(y_t - \hat{A}_t,t) + \hat{a}_t)dt
\nonumber \\
dW_t^{\bar{a}} &= dW_t^0 - \frac{1}{\sigma}(\hat{\eta}(y_t - \bar{A}_t,t) + \bar{a}_t)dt
\nonumber \\
dW_t^{\hat{a}} &= dW_t^{\bar{a}} + \frac{1}{\sigma} [ \tilde{\eta}(y_t - \bar{A},t) + \bar{a}_t - \tilde{\eta}(y_t - \hat{A},t) - \hat{a}_t ]dt
\end{align}
Hence, the following holds
\begin{align}
&v(t,\bar{a}) - v(t,\hat{a}) = e^{\rho t}\mathbb{E}_{\bar{a}}\bigg[ \int_t^T e^{-\rho(s-t)}[u(w_t,\bar{a}_t) - u(w_t,\hat{a}_t)]dt + 
\nonumber \\ &\indeq 
+ \int_t^T \hat{\gamma}_t \sigma dW_t^{\bar{a}} 
\nonumber \\ &\indeq 
+ \int_t^T \hat{\gamma}(\bar{a}_t - \hat{a}_t + \tilde{\eta}(y_t,\bar{A}_t) - \tilde{\eta}_t(y_t,\hat{A}_t))dt  \bigg]
\nonumber \\ &\indeq 
= e^{\rho t}\mathbb{E}_{\bar{a}}\bigg[ \int_t^T [H(\bar{a},\hat{\gamma}) - H(\hat{a},\hat{\gamma}) ] dt + \int_t^T \hat{\gamma}_t \sigma dW_t^{\bar{a}}
\nonumber \\ &\indeq 
\leq e^{\rho t}\mathbb{E}_{\bar{a}} \bigg[ \int_t^T \hat{\gamma}_t \sigma dW_t^{\bar{a}} \bigg] = 0.
\end{align}
The last term is a martingale due to square integrability of $\hat{\gamma}_t$ and $a \in [0,M]$ being bounded. Hence, we have proved the sufficient condition for the agent.
Next, we prove the necessary condition for the agent's value function in Equation $\eqref{eqn38}$.
\begin{align*}
&\tilde{a}_t^\epsilon  \triangleq a_t + \epsilon\Delta a_t
\\ \nonumber
&  \nabla v_t(a) \triangleq \lim_{\epsilon \rightarrow 0}\frac{v_t(\tilde{a}) - v_t(a)}{\epsilon} 
\end{align*}
and by small perturbation, we have 
\begin{align} 
&e^{-\rho t}\nabla v_t(a) = e^{-\rho t}\lim_{\epsilon \rightarrow 0} \frac{1}{\epsilon} \mathbb{E}_{\tilde{a}}\bigg[ \int_t^T e^{-\rho s}[ u(w,\tilde{a}^\epsilon) - u(w,a)]ds 
\nonumber \\ &\indeq
+ \int_t^T \hat{\zeta}_sdW_s^{\bar{a}^\epsilon} + \int_t^T \hat{\zeta}( \tilde{a}^\epsilon - a + \hat{\eta}(y_s,A_s^\epsilon) - \hat{\eta}(y_s,A_s) )ds
\bigg],
\end{align}
which gives the condition 
\begin{equation}
\mathbb{E}_{\tilde{a}}\bigg[ \int_t^T e^{-\rho s}u_a\Delta a_s + \zeta_s \bigg(\Delta a_s -  \frac{\sigma^{-2}}{h_s}\int_t^s \Delta a_r dr \bigg)  \bigg] \leq 0,
\end{equation}
Integrating by parts, we get that
\begin{equation*}
\mathbb{E}_{\tilde{a}}\bigg[ \int_t^T \bigg( e^{-\rho s}u_a + \zeta_s -  \int_s^T \zeta_r \frac{\sigma^{-2}}{h_r} dr \bigg) \Delta a_s  ds \bigg] \leq 0,
\end{equation*}
By noting that $\Delta a_s$ is arbitrary, we get 
\begin{equation}
\bigg( \mathbb{E}_t^a \bigg[ \int_t^T -\zeta_s \frac{\sigma^{-2}}{h_s} ds  \bigg] + \zeta_t + e^{-\rho s}u_a(w_s,a_s ) \bigg)( a_t - a^*_t) \leq 0
\end{equation}
By focusing only at time $t$ and using $\Delta a_t$ is arbitrary, we conclude that  for\\ $a_t \in [0,M]$, we have
\begin{equation}
\label{eqn332}
\zeta_t + e^{-\rho t}u_a(w_t,a^*_t) - \frac{\sigma^{-2}}{h_t}p_t \geq 0,
\end{equation}
where 
\begin{equation}
\label{eqn333}
p_t = h_t \mathbb{E}\bigg[ - \int_t^T \zeta_s \frac{1}{h_s}ds \bigg| \mathcal{F}_t^a \bigg].
\end{equation}
Since increasing $\gamma_t^a$ causes the volatility of the output to increase in Equation $\eqref{decomposition}$, the principal wants to minimize the incentive compatibility parameter $\gamma_t^a$. Hence, by multiplying the equation by $e^{\rho t}$ in Equation \eqref{eqn332}, we assume that the principal confines with 
\begin{equation}
\label{eqn334}
\gamma_t^a(a,w)= - u_a(w_t,a^*_t) + \frac{\sigma^{-2}}{h_t}p_t,
\end{equation} 
and conclude the result.
\end{proof}

We rewrite the term $p_t$ in Equation \eqref{eqn332} in a more tractable way as in \cite{PJ15} as follows. First, we denote 
\begin{equation*}
\tilde{p}_t \triangleq \frac{\sigma^{-2}}{h_t}p_t
\end{equation*}
Then, we have 
\begin{equation*}
\tilde{p}_t = \mathbb{E}\bigg[ -\int_t^T e^{-\rho(s-t)}\zeta_s \frac{\sigma^{-2}}{h_s}ds\bigg],
\end{equation*}
By differentiating with respect to time $t$, we get 
\begin{align*} 
\frac{d\tilde{p}_t}{dt} &= \rho \tilde{p}_t + \frac{\sigma^{-2}}{h_t}\gamma_t
\\ 
&= \rho \tilde{p}_t - \frac{\sigma^{-2}}{h_t}(u_a(w_t,a_t) + \tilde{p}_t ),
\end{align*}
Integrating this expression, we obtain 
\begin{equation}
\label{eqn335}
\tilde{p}_t = \frac{\sigma^{-2}}{h_t}\mathbb{E}_a\bigg[ \int_t^T e^{-\rho(s-t)}u_a(w_s,a_s) ds \bigg]
\end{equation}
\begin{Remark}
By the above derivation we see that $\gamma_t(w,a)$ is bounded by wage process $w_t$ being non-negative and $a \in [0,M]$. Furthermore, we also note that when there is no term $p_t$ in Equation \eqref{eqn334}, we have $\gamma_t = -u_a$, which corresponds to the first order condition with respect to $a$ of the term
\begin{equation}
\tilde{H}(t,y,a,\gamma) = u(w,a) + \gamma_t a_t,
\end{equation}
the Hamiltonian term without the filtering term in the model Equation \eqref{main_eqn}.
\end{Remark}
For the terminal date, we assume that from date $T$ on the unknown filtering term $\hat{\eta}$ is revealed, no more production takes place and both the principal and agent live off their assets for the infinite future, earning the same constant rate of return $r$. We assume both the principal and the agent solve the problem of the following form
\begin{align} 
\label{eqn336}
V_T(a_0) &= \max_{b_t} - \int_0^\infty \exp(-\rho t - \lambda \theta b_t)dt
\nonumber \\
&= \max_{b_t} - \int_T^\infty \exp(-\rho (t-T) - \lambda \theta b_t)dt,
\end{align}
with $c_0$ given and $dc_t = (rc_t - b_t)dt$. For the agent $b_t = w_t$ and $a_0 = w_T$. The Hamilton-Jacobi-Bellman (HJB) equation for \eqref{eqn336} reads as 
\begin{equation}
\rho V_T(a) = \max_b\{ -\exp(-\lambda \theta b) + V^\prime_T(c)[rc - b] \},
\end{equation}
whose solution is 
\begin{equation*}
    V_T(c) = -\exp \big( \frac{1-\rho}{r} - \lambda \theta r c \big),
\end{equation*}
with optimal 
\begin{equation*}
b(c) = \frac{\rho - 1}{\lambda \theta r} + rc.
\end{equation*}
Hence, we have for the terminal time $T$, the agent's and principal's terminal value function $v_T$ and $V_T^p$ as 
\begin{align}
\label{eqn100}
g(w_T) &= -\exp\big( \frac{1-\rho}{r} - \lambda \theta r w_T \big)
\\
v_T &= g(w_T) 
\\
V_T^p &= g(y_T - w_T).
\end{align}

\section{Principal's Problem}
From the principals point of view, the dynamics of the output follows
\begin{equation}
dy_t = (ry_t + \hat{\eta} + a_t  -d_t)dt + \sigma dW_t^a.
\end{equation}
We assume there is a common risk aversion $\lambda$ between the principal and the agent. The principal discounts at the same rate $\rho$ with the agent and has a flow utility 
\begin{equation}
U(d_t)= -\exp(-\lambda \theta d_t)
\end{equation}
over his consumption $d_t$ with the value function
\begin{equation}
\label{eqn442}
J(t,y,v,\hat{\eta}) = \max_{d,w,a} \mathbb{E}_t^a\bigg[ \int_t^T e^{-\rho(s-t)}U(d_s)ds + e^{-\rho(T-t)}V^p_T(y_T - w_T) | \mathcal{F}_t^a \bigg],
\end{equation}
where $V^p_T(y_T - w_T)$ is defined as follows.
\begin{align*}
V^p_T(y_T - w_T) &= -\exp\big( \frac{1-\rho}{r} - \lambda \theta r (y_T - w_T) \big)
\\
&= -\exp\big( \frac{1-\rho}{r} - \lambda \theta ry_T  \big) \exp( \lambda \theta r w_T ),
\end{align*}
Using the terminal value of the agent at time $T$, the principal value function at time $T$ reads as using Equation \eqref{eqn100}
\begin{align*}
    v_T &= -\exp \big( \frac{1-\rho}{r} - \lambda \theta r w_T \big)
    \\
    J(T,y_T,v_T) &= -\frac{\exp^2( \frac{1-\rho}{r} )}{v_T}\exp(-\lambda \theta r y_T)
\end{align*}
For convenience, we summarize the value function dynamics of the principal as follows:
\begin{align}
\label{eqn442}
d\hat{\eta} &= \frac{\sigma^{-1}}{h_t}dW_t^a \\
\hat{\eta}(y_t - A_t, t) &= \frac{h_0m_0 + \sigma^{-2}(y_t-A_t)}{h_t} \\
\eta(0,0) &= m_0 \\
dv_t &= [\rho v_t - u(w_t,a_t) ]dt + \sigma\gamma_t^a(a,w)dW_t^a \\
v_T &= -\exp\big( \frac{1-\rho}{r} \big)\exp(-\lambda \theta r w_T) \\
dy_t &= \big( r y_t + \hat{\eta} + a_t - d_t \big)dt + \sigma dW_t^a \\
\label{eqn448}
y_0 &= 0
\end{align}
We define the controlled value function for fixed admissible action process $a_t$ as
\begin{equation}
J^u(t,v,y,\hat{\eta}) = \mathbb{E}_{t,T}^a[\int_t^T e^{-\rho(s-t)}e^{-\lambda \theta d_s}ds + V_T^p(y_T - w_T)| \mathcal{F}_t^a]
\end{equation}
and the value function of the control problem given $(t,\hat{\eta},v,y) \in [0,T] \times \mathbb{R}^3$ as, 
\begin{equation}
\label{eqn450}
J(t,\hat{\eta},v,y) := \sup_{a(\cdot) \in \mathcal{A}}J^a(t,\hat{\eta},v,y)
\end{equation}
We next state our main theorem in this section and prove it in the subsection below, subsequently. 
\begin{Theorem}
\label{thm41}
Suppose that the principal and the agent with an unknown quality term have the value functions Equation \eqref{eqn442} and Equation \eqref{eqn38}, respectively, then a contract is implementable in the sense of Definition \ref{def21}, where both parties agree to recommend and give full effort for all times $0\leq t \leq T$.
\end{Theorem}
\subsection{Proof of Theorem \ref{thm41}}
To prove Theorem \ref{thm41}, we guess an explicit solution for the value function in Equation \eqref{eqn450} and verify our guess subsequently. 
Next, we guess that for $0 \leq t \leq T$ the value function of the principal is a $\mathcal{C}^{1,2}$ function of the form
\begin{equation}
J(t,y,v,\hat{\eta}) = \frac{e^{g(t,\hat{\eta})}}{v}\exp(-\lambda \theta r y)
\end{equation} and verify its validity below.
We further guess that for the optimal action process $a^*_t$
\begin{equation}
e^{-\theta w + \theta \lambda a^*_t} = k(t,\hat{\eta})v
\end{equation}

Then by Theorem \ref{hamilton}, we have that
\begin{align}
dp_t &= \theta\lambda dv_t
\\
p_t &= \theta \lambda \mathbb{E}^a\bigg[ \int_t^T e^{-\rho(s-t)}u ds \bigg]
\\
p_t &= \theta \lambda [1 - e^{\int_t^T (\rho - k(s,\eta))ds}]v_t
\\
\varphi_t(k) &= 1 - e^{\int_t^T (\rho - k(s,\eta))ds}
\end{align} 
Furthermore, using our guess for the value function, we obviously have
\begin{align*} 
J_y &= -\lambda \theta r J \\
J_{yy} &= \lambda^2 \theta^2 r^2 J \\
J_v &= -\frac{1}{v}J \\
J_{yv} &= \lambda \theta r \frac{1}{v} J \\
J_{vv} &= \frac{2}{v^2}J
\end{align*}
Following our guess for the value function being in $\mathcal{C}^{1,2}$, the Hamilton-Jacobi-Bellman (HJB) equation is of the form
\begin{align}
\label{eqn102}
\rho J - J_t &= \max_{w,d,a} \bigg\{  -\exp(-\lambda \theta d) + J_y [ry + \eta + a - w - d]
\nonumber \\&\indeq 
+ J_v [ \rho v + e^{-\theta w + \lambda \theta a} ]
\nonumber \\&\indeq
+ \frac{1}{2}J_{yy}\sigma^2 + \frac{1}{2}J_{vv}\sigma^2\gamma_t^2
\nonumber \\&\indeq
+ \frac{1}{2}J_{\eta\eta}\frac{\sigma^{-2}}{h_t} + J_{y\eta}\frac{1}{h_t}
\nonumber \\&\indeq
+ J_{yv}\sigma^2\gamma_t + J_{v\eta}\frac{\gamma_t}{h_t}
\bigg \}
\end{align}
Next, we show that our guess value function necessitates that the principal advises full action, namely $a^*_t \equiv M$ for $0 \leq t \leq T$.
\begin{Lemma}\label{guessLem} The recommended action is $a^* = M$, namely the right corner is optimal for the principal.
\end{Lemma}
\begin{proof} By writing the first order condition for wage $w$ and action $a$ from the HJB, we have the following pair of equations:
\begin{align} 
\label{eqn476}
-J_y + J_v[-\theta e^{-\theta w + \lambda\theta a }] + \frac{1}{2}J_{vv}\sigma^2 2 \gamma \gamma_w + J_{yv}\sigma^2 \gamma_w + J_{v\eta} \frac{\gamma_w}{h_t} &= 0
\\
\label{eqn477}
J_y + J_v[ \lambda \theta e^{-\theta w + \lambda\theta a } ] + \frac{1}{2}J_{aa}\sigma^2 2 \gamma \gamma_a + J_{yv}\sigma^2 \gamma_a + J_{v\eta} \frac{\gamma_a}{h_t} &= 0
\end{align}
Then, our guess for the value function 
\begin{align}
J &= e^{g(t,\eta)}\frac{e^{-\lambda \theta r y}}{v}
\\
J_y &= -\lambda \theta r e^{g(t,\eta)}\frac{e^{-\lambda \theta r y}}{v},
\end{align}
we see that $J_y$ is positive, since $v$ is negative.By the relation above,\\ $\frac{\partial \gamma}{\partial w} = -\lambda\frac{\partial \gamma}{\partial a} > -\frac{\partial \gamma}{\partial a}$. Hence, by noting $0<\lambda <1$ and by derivatives of the exponential function with respect to $a$ and $w$, the first order condition for $w$ binds, whereas the first order condition for $a$ does not bind. So we have either $a=0$ or $a=M$ as optimal actions. But for $a=0$ to be optimal, the right derivative should be less than or equal to 0 at $a= 0$, but this can not be the case due to first order condition for $a$ and $w$ above. Similarly, the right corner's left derivative is positive, whenever the first order condition for $w$ binds, hence  optimal action $a^*_t = M$ for all $t \in [0,T]$.
\end{proof}

Using our guesses for the utility function and the value function and suppressing the arguments of the functions for simplicity below, we rewrite the HJB equation as follows 
\begin{align*}
\rho J - J_t &= \max_{w,d}\bigg\{ -e^{-\lambda \theta d} -\lambda \theta r J [ ry + \eta + \frac{\log(kv)}{\lambda \theta} + \frac{\log(\lambda \theta r e^{g(t,\eta)})}{\lambda \theta} -\frac{\log(-v)}{\lambda \theta} - ry ]
\nonumber \\&\indeq
-\frac{1}{v}J[\rho v + k v] + \frac{1}{2}\lambda^2 \theta^2 r^2 J \sigma^2 + \frac{1}{2}\frac{2}{v^2}\sigma^2 \theta^2 \lambda^2 \big( k + \frac{\sigma^{-2}}{h_t}\varphi \big)^2v^2 
\nonumber \\&\indeq
\frac{1}{2}J_{\eta\eta}\frac{\sigma^{-2}}{h_t} - \lambda \theta r J_{\eta}\frac{1}{h_t}
\nonumber \\&\indeq
+\lambda \theta r \frac{1}{v}J\sigma^2\lambda \theta [k + \frac{\sigma^{-2}}{h_t}\varphi]v - \frac{1}{v}J_{\eta}\frac{1}{h_t}\lambda \theta v [k + \frac{\sigma^{-2}}{h_t}\varphi]
\bigg\} 
\end{align*}
By cancelling the terms and by first order condition on $d$, i.e. $e^{-\lambda \theta d} = -rJ$, we get
\begin{align} 
\label{eqn477}
\rho J - J_t &= \bigg\{ rJ - \lambda \theta J [\eta + \frac{\log (-k)}{\lambda \theta} + \frac{\log(\lambda \theta r e^{g(t,\eta)})}{\lambda \theta}]
\nonumber \\&\indeq 
-J[\rho + k] + \sigma^2\theta^2\lambda^2[k + \frac{\sigma^{-2}}{h_t}\varphi]^2J
\nonumber \\&\indeq
+\frac{1}{2}J_{\eta\eta}\frac{\sigma^{-2}}{h_t} - \lambda \theta r J_\eta\frac{1}{h_t}  
\nonumber \\&\indeq 
+ \lambda^2 \theta^2 \sigma^2 r J [k + \frac{\sigma^{-2}}{h_t}\varphi] - \lambda \theta  \frac{1}{h_t}[k + \frac{\sigma^{-2}}{h_t}\varphi]J_\eta 
\bigg\}
\end{align} 
By our guess for the value function, we have
\begin{align*}
J &= \frac{e^{g(t,\eta)}}{v}e^{-\lambda \theta r y}
\\
J_t &= g_t J
\\
J_{\eta} &= g_\eta J
\end{align*}
Hence, the HJB Equation \eqref{eqn477} reads as 
\begin{align}
\label{guess_HJB}
\rho - g_t &=  r - \lambda \theta \bigg[\eta + \frac{\log (-k)}{\lambda \theta} + \frac{\log(\lambda \theta r)}{\lambda \theta} + \frac{g(t,\eta)}{\lambda \theta} \bigg]
\nonumber \\&\indeq 
-[\rho + k] + \sigma^2\theta^2\lambda^2[k + \frac{\sigma^{-2}}{h_t}\varphi(k)]^2
\nonumber \\&\indeq
+\frac{1}{2}g_{\eta\eta}\frac{\sigma^{-2}}{h_t} - \lambda \theta r g_\eta\frac{1}{h_t}  
\nonumber \\&\indeq 
+ \lambda^2 \theta^2 \sigma^2 r  [k + \frac{\sigma^{-2}}{h_t}\varphi(k)] - \lambda \theta  \frac{1}{h_t}[k + \frac{\sigma^{-2}}{h_t}\varphi(k)]g_\eta 
\end{align}
with the terminal condition $g(T,\hat{\eta}) = e^{2(\frac{1-\rho}{r})}$ for all $\hat{\eta} \in \mathbb{R}$.
\newline\newline
For ease of notation, we introduce the following expressions. 
\begin{align}
K_1(t,\hat{\eta}) &= -\rho + r - \lambda \theta \hat{\eta} - \log(-k(t,\hat{\eta})) - \log(\lambda \theta r) - \rho -k(t,\eta) \\
&+ \sigma^2\theta^2\lambda^2[k(t,\hat{\eta}) + \frac{\sigma^{-2}}{h_t}\varphi(k(t,\hat{\eta}))]^2 + \lambda^2\theta^2\sigma^2r [k(t,\hat{\eta})] \\
K_2(t,\hat{\eta}) &= -\lambda \theta r \frac{1}{h_t} - \lambda \theta\frac{1}{h_t}[k(t,\hat{\eta}) + \frac{\sigma^{-2}}{h_t}\varphi(k(t,\hat{\eta}))]\\
K_3(t) &= \frac{\sigma^{-1}}{h_t}
\end{align}
By Feynman-Kac, the existence and uniqueness of the PDE above is guaranteed as 
\begin{equation}
g(t,\hat{\eta}) = \mathbb{E}_t^a\bigg[ \int_t^T e^{-(T-s)}K_1(s,\hat{\eta}_s)ds + e^{-(T-t)}e^{2\frac{1-\rho}{T}} | \mathcal{F}_t^a\bigg],
\end{equation}
under the action $a$ such that $\hat{\eta}$ is an Ito process driven by the equation \begin{equation}
d\hat{\eta} = K_2(t,\hat{\eta})dt + K_3(t,\hat{\eta}_t)dW_t^a
\end{equation} 
Moreover, using our guesses for the value function and utility function, we rewrite the first order condition for $w$ as 
\begin{align}
\label{eqn101}
&\lambda \theta r J - \frac{1}{v}[-\theta kv]J + \frac{1}{2}\frac{2}{v^2}J\sigma^2 2 \theta \lambda v [ k + \frac{\sigma^{-2}}{h_t}\varphi](-\theta)^2\lambda^2kv 
\\ & \indeq
+\lambda \theta r \frac{1}{v}J\sigma^2\theta\lambda v [k + \frac{\sigma^{-2}}{h_t}\varphi] + g_{\eta}\frac{1}{v}J\theta^2\lambda^2\frac{kv}{h_t} = 0
\end{align}
Hence by cancelling $v$ from the equation and dividing by $J$, the Equation \eqref{eqn101} reads as
\begin{align}
\label{k_eqn}
&\lambda \theta r + \theta k - \sigma^2 2 \theta \lambda[k + \frac{\sigma^{-2}}{h_t}\varphi]\theta^2\lambda^2k
\\& \indeq
+ \lambda \theta r \sigma^2\theta \lambda [k + \frac{\sigma^{-2}}{h_t}\varphi] + g_\eta\theta^2 \lambda^2\frac{k}{h_t} = 0.
\end{align}
Equation \eqref{guess_HJB} and Equation \eqref{k_eqn}, $k(t,\hat{\eta})$ and $g(t,\hat{\eta})$ are implicitly defined and can be found numerically. 
\subsection{Verification Theorem}
By the discussion above, we have the following converse relation. Our guess 
\begin{equation}
\label{eqn476}
J(t,y,v,\hat{\eta}) = \frac{e^{g(t,\hat{\eta})}}{v}\exp(-\lambda \theta r y)
\end{equation}
is a $\mathcal{C}^{1,2}$ function. It also satisfies the HJB equation \eqref{eqn102}
\begin{align*}
\rho J - J_t &= \max_{w,d,a} \big\{  -\exp(-\lambda \theta d) + J_y [ry + \eta + a - w - d]
\nonumber \\&\indeq 
+ J_v [ \rho v + e^{-\theta w + \lambda \theta a} ]
\nonumber \\&\indeq
+ \frac{1}{2}J_{yy}\sigma^2 + \frac{1}{2}J_{vv}\sigma^2\gamma_t^2
\nonumber \\&\indeq
+ \frac{1}{2}J_{\eta\eta}\frac{\sigma^{-2}}{h_t} + J_{y\eta}\frac{1}{h_t}
\nonumber \\&\indeq
+ J_{yv}\sigma^2\gamma_t + J_{v\eta}\frac{\gamma_t}{h_t}
\big \}
\end{align*} with boundary condition 
\begin{equation}
 J(T,y_T,v_T) = -\frac{\exp^2( \frac{1-\rho}{r} )}{v_T}\exp(-\lambda \theta r y_T).
\end{equation} 
Using our guesses for the utility function we also have by the discussion above for optimal action $a^*$
\begin{align*}
e^{-\theta w + \theta \lambda a^*} &= k(t,\eta) v
\\
-\lambda \theta w + \lambda \theta a^* &= \log (kv)
\\
-\lambda \theta w &= \log (kv) - \lambda \theta a^* 
\\
-w &= \frac{\log kv}{\lambda \theta} - M
\\
a^* - w = M - w &= \frac{\log(kv)}{\lambda \theta}
\\
-e^{-\lambda \theta d} &= \lambda \theta r J
\\ 
e^{-\lambda \theta d} &= -r \frac{e^{g(t,\hat{\eta})}}{v}e^{-\lambda \theta r y}
\\
-\lambda \theta d &= \log (rg(t,\hat{\eta})) - \log(-v) - \lambda \theta r y 
\\
-d &= \frac{\log(rj^1)}{\lambda\theta} - \frac{\log(-v)}{\lambda \theta} - ry
\end{align*}
Hence, for each fixed $(t,y,v,\hat{\eta})$
the expression 
\begin{align}
&\max_{w,d,a} \big\{  -\exp(-\lambda \theta d) + J_y [ry + \eta + a - w - d]
\nonumber \\&\indeq 
+ J_v [ \rho v + e^{-\theta w + \lambda \theta a} ]
\nonumber \\&\indeq
+ \frac{1}{2}J_{yy}\sigma^2 + \frac{1}{2}J_{vv}\sigma^2\gamma_t^2
\nonumber \\&\indeq
+ \frac{1}{2}J_{\eta\eta}\frac{\sigma^{-2}}{h_t} + J_{y\eta}\frac{1}{h_t}
\nonumber \\&\indeq
\label{eqn478}
+ J_{yv}\sigma^2\gamma_t + J_{v\eta}\frac{\gamma_t}{h_t}
\big \}
\end{align} attains its maximum $(a^*,w^*,d^*)$ at 
\begin{align}
\label{eqn480}
a^* &= M \\
\label{eqn481}
w^* &= M - \frac{\log k(t,\hat{\eta}^{a^*})v}{\lambda \theta} \\
\label{eqn482}
d^* &= yr - \frac{\log(-v)}{\lambda \theta} - \frac{\log(r)}{\lambda \theta} - \frac{g(t,\hat{\eta}^{a^*})}{\lambda \theta},
\end{align} where the functions $g(t,\hat{\eta})$ and $k(t,\hat{\eta})$ are determined by Equation $\eqref{k_eqn}$ and Equation $\eqref{guess_HJB}$. Hence, Equation \eqref{eqn476} is indeed the solution of the HJB Equation \eqref{eqn478} and we conclude the verification theorem. 

\section{Discussion and Conclusion}
In this paper, we have studied a principal-agent problem with moral hazard. Contrary to the mainstream, where the principal is assumed to be risk-neutral, we have assumed that both the principal and the agent have exponential utility as in Equation \eqref{eqn25}, and they are risk-averse with the same risk awareness level $\lambda$ as in \cite{N14}. We take also an unknown endogenous learning term representing the unknown quality of the agent into account, which is revealed as time passes by. We see that both parties agree on a contract, where the agent gives full effort from beginning until the finite horizon $T$. Full effort of the agent in the optimal contract is also observed in \cite{PJ15}, where the authors assumed the principal to be risk-neutral. Hence, we see that the risk-awareness level of the principal does not affect a role in the agent's actual effort in the contract, but only the agent's utility is determinant in that respect. The optimality of the right corner of the admissible action interval of the agent as to be seen in Lemma \ref{guessLem} is due to the specific nature of the utility function chosen in Equation \eqref{eqn25}. Changing the utility function of the agent would cause the arguments in Lemma \ref{guessLem} would not hold anymore. In that case, one usually \textit{assumes} that the optimal effort is in the interior of the effort interval. On the other hand, we see that the payments to the agents in terms of wages and dividends are affected by the posterior belief on the unknown quality of the agent $\hat{\eta}$, as well as on the risk-awareness $\lambda$ of the principal as to be seen in the Equations \eqref{eqn480}, \eqref{eqn481} and \eqref{eqn482}. However, we can not conclude that there is a direct negative effect on the payment to the agent due to the nonlinear nature of the parameters. We further note that since the uncertainty on quality of the agent decreases as to be seen in Equation \eqref{eqn22}, its effect on the dynamics of the problem decreases as time $t$ passes.

\end{document}